\documentclass[12pt, reqno]{amsart}

\usepackage{amsmath, amsthm, amssymb, stmaryrd}
\usepackage{hyperref}
\usepackage{enumerate}
\usepackage{url}

\usepackage{xcolor}

\usepackage[centering, margin={1in, 0.5in}, includeheadfoot]{geometry}

\usepackage{fancyvrb}

\newtheorem{thm}{Theorem}
\newtheorem{lemma}[thm]{Lemma}

\newtheorem{conj}[thm]{Conjecture}

\theoremstyle{definition}

\theoremstyle{remark}

\def\alp{{\alpha}} 
\def\bet{{\beta}}  
\def\gam{{\gamma}} 
\def\del{{\delta}} 

\def\sig{{\sigma}}

\def\le{\leqslant} \def\ge{\geqslant}

\def\d{{\,{\rm d}}}

\def \sig{{\sigma}}

\def \bN {\mathbb N}
\def \N {\mathbb N}

\def \bR {\mathbb R}
\def \R  {\mathbb R}
\def \bZ {\mathbb Z}

\def \cH {\mathcal H}

\def \dim {\mathrm{dim}}

\DeclareMathOperator{\dimh}{\dim_H}

\begin{document}
\title{A note on dyadic approximation in Cantor's set}
\author{Demi Allen, Simon Baker, Sam Chow \and Han Yu}

\address{College of Engineering, Mathematics and Physical Sciences, 
University of Exeter, Harrison Building, 
North Park Road, Exeter, EX4 4QF, United Kingdom}
\email{d.d.allen@exeter.ac.uk}

\address{School of Mathematics, Watson Building, University of Birmingham, Edgbaston,
Birmingham B15 2TT,
United Kingdom}
\email{s.baker.2@bham.ac.uk}

\address{Mathematics Institute, Zeeman Building, University of Warwick, Coventry CV4 7AL, United Kingdom}
\email{Sam.Chow@warwick.ac.uk}

\address{Department of Pure Mathematics and Mathematical Statistics, Centre for Mathematical Sciences, Cambridge, CB3 0WB, UK}
\email{hy351@cam.ac.uk} 

\begin{abstract} 
We consider the convergence theory for dyadic approximation in the middle-third Cantor set, $K$, for approximation functions of the form $\psi_{\tau}(n) = n^{-\tau}$ ($\tau \ge 0$). In particular, we show that for values of $\tau$ beyond a certain threshold we have that almost no point in $K$ is dyadically $\psi_{\tau}$-well approximable with respect to the natural probability measure on~$K$. This refines a previous result in this direction obtained by the first, third, and fourth named authors (arXiv, 2020).
\end{abstract}

\maketitle
\thispagestyle{empty}

\section{Introduction}

Throughout this note, we write $K$ for the middle-third Cantor set and denote by $\mu$ the natural measure on $K$. We recall that $K$ consists of the real numbers $x \in [0,1]$ which have a ternary expansion consisting only of 0's and 2's, and that its Hausdorff dimension is 
\[\dimh{K} = \frac{\log{2}}{\log{3}} =: \gamma.\]
The natural measure $\mu$ on $K$ is the Hausdorff $\gamma$-measure restricted to $K$, which is a probability measure as $\cH^{\gamma}{(K)}=1$. For more information on Hausdorff dimension and Hausdorff measures, we refer the reader to \cite{Falconer}.

In \cite{ACY}, the first, third and fourth named authors discussed the problem of approximating elements of $K$ by rationals with denominators that are a power of two: that is, \emph{dyadic rationals}. The study of Diophantine approximation in the Cantor set was suggested by Mahler \cite{Mah1984}, and has since been an active subject of research \cite{Bug2008, BD2016, KL, LSV2007, Sch2021, SW2019, Wei2001, Yu}. Our methods realised the dyadic approximation problem as a manifestation of Furstenberg's ``times two, times three'' phenomenon \cite{Furstenberg67,Furstenberg70}.

For $\psi: \bR \to [0,\infty)$ and $y \in \bR$, define
\[
W_2(\psi,y) = \{ x \in \bR: \| 2^n x - y \| < \psi(n) \: \text{for infinitely many } n \in \N \}.
\]
Here, for $x \in \R$, we write $\|x\|$ to denote the Euclidean distance from $x$ to the nearest integer. In analogy with Khintchine's theorem \cite{Khintchine24}, Velani conjectured that if $\psi$ is monotonic then
\[
\mu(W_2(\psi,0)) = \begin{cases}
0,&\text{if }
\displaystyle \sum_{n = 1}^\infty \psi(n) < \infty, \\
\\
1, &\text{if }
\displaystyle \sum_{n = 1}^\infty \psi(n) = \infty,
\end{cases}
\]
see \cite[Conjecture 1.2]{ACY}.
The two parts of such a dichotomy are commonly referred to as the convergence and divergence theories of metric Diophantine approximation, respectively. The second named author \cite{Bak} stated the following natural generalisation of Velani's conjecture, dropping the monotonicity condition and introducing an inhomogeneous shift. The latter relates the problem to distribution modulo $1$, and also enables one to recast it in terms of shrinking targets \cite{HV1995}.

\begin{conj}[\cite{Bak}] If $y \in \bR$, then
\[
\mu(W_2(\psi, y)) = \begin{cases}
0,&\text{if } \displaystyle \sum_{n = 1}^\infty \psi(n) < \infty, \\
\\
1, &\text{if } \displaystyle \sum_{n = 1}^\infty \psi(n) = \infty.
\end{cases}
\]
\end{conj}

Let us now consider the problem at the level of the exponent. For $\tau \ge 0$ and $n \in \bN$, define
$\psi_\tau(n) = n^{-\tau}$. Plainly $\mu(W_2(\psi_0,y)) = 1$ for any $y$. By \cite[Theorem 1.5]{ACY}, we have 
\begin{equation} \label{PreviousConv}
\mu(W_2(\psi_\tau,0)) = 0 \qquad (\tau \ge 1/\gam).
\end{equation}
It follows from the recent work of the second named author \cite{Bak} that if $y \in \bR$ then
\[
\mu(W_2(\psi_\tau,y)) = 1 \qquad (\tau \le 0.01),
\]
refining the progress on the divergence side made in \cite{ACY}. The purpose of this note is to establish the following sharpening and generalisation of \eqref{PreviousConv}.

\begin{thm} \label{MainThm} Let $\tau \ge 1/\gam - 0.01$ and let $y \in \bR$. Then $\mu(W_2(\psi_\tau,y)) = 0$.
\end{thm}

This makes progress towards the convergence part of Velani's conjecture. In \cite{ACY}, it was shown conditionally that
\[
\mu(W_2(\psi_\tau,0)) = 
\begin{cases}
0, &\text{if } \tau > 1, \\
1, &\text{if } \tau \le 1,
\end{cases}
\]
which constitutes a conditional solution to Velani's conjecture at the level of the exponent. Theorem \ref{MainThm} is unconditional.

We finish this section by briefly discussing the significance of the exponent $1/\gam$. By a comparatively simple argument, one can see that if $\tau > 1/\gam$ and $y \in \bR$ then $\mu(W_2(\psi_\tau, y)) = 0$, see the proof of \cite[Proposition 1.4]{ACY}. In \cite{ACY}, we attained the exponent $1/\gam$ in establishing \eqref{PreviousConv}. Thus, as explained in the introduction of that article, dyadic approximation in $K$ behaves very differently to triadic approximation in $K$, the latter having been thoroughly investigated by Levesley, Salp and Velani \cite{LSV2007}. Theorem \ref{MainThm} extends the admissible range for the exponent beyond this threshold.

\subsection*{Notation.} For complex-valued functions $f$ and $g$, we write $f \ll g$ or $f = O(g)$ if $|f| \le C |g|$ pointwise, for some constant $C > 0$.

\subsection*{Funding.} Han Yu was supported by the European Research Council under the European Union’s Horizon 2020 research and innovation programme (grant agreement No. 803711), and indirectly by Corpus Christi College, Cambridge.

\section{Proof of Theorem \ref{MainThm}}

We now prove Theorem \ref{MainThm}. 
Set $C>0$ to be the constant $C_1$ arising from \cite[Lemma 2.2]{Bak}, and let $N \in \bN$ be large. For $n \in [N,2N] \cap \bZ$, put
\[
\sig_n = n^{-\tau},
\qquad
\del_n = n^{-\alp},
\]
where
\[
\alp = 0.05, \qquad \bet_1 = 0.078,
\qquad \bet_2 = 0.922
\]
and
\begin{equation} \label{constraint}
\tau \gam > \max \{1 - \alp(1-\gam), \bet_2 + \alp \}.
\end{equation}
For $k \in \bZ$, denote
\[
\hat \mu(k) = \int_0^1 e(-kx) \d \mu(x).
\]
Write $G_N$ for the set of integers $n \in [N,2N]$ such that
\[
\max_{1 \le |t| \le 2/\del_{2N}} |\hat \mu(t 2^n)| \le CN^{-\bet_1},
\]
and let $B_N$ be its complement in $[N,2N] \cap \bZ$. By \cite[Lemma 2.2]{Bak} 
and the union bound, we have 
\[
|B_N| \ll N^{\bet_2 + \alp}.
\]

For $n \in \bN$ and $\sig > 0$, denote
\[
A_n^y(\sig) = \{ x \in \bR: \| 2^n x - y \| < \sig \},
\]
so that
\[
W_2(\psi_\tau,y)) = \limsup_{n \to \infty} A_n^y(\sig_n).
\]
By the first Borel--Cantelli lemma \cite[Lemma 1.2]{Harman}, it suffices to prove that 
\begin{equation} \label{BC}
\sum_{n=1}^\infty \mu(A_n^{y}(\sig_n)) < \infty.
\end{equation}
For $n \in B_N$, we use the following estimate, whose proof follows straightforwardly from the argument in \cite[\S 2.1]{ACY}.

\begin{lemma} Let $y \in \bR$. Then
\[
\mu(A_n^y(\sig_n)) \ll \sig_n^\gam \qquad (n \in \bN).
\]
\end{lemma}

In the generic case $n \in G_n$, we are able to procure a stronger estimate by transferring data from the coarse scale $\del_n$ to the fine scale $\sig_n$. By \cite[Theorem 4.1]{Yu}, we have
\[
\mu(A_n^y(\del_n)) \ll \del_n \left(
1 + \sum_{ 1 \le |t| \le 2/\del_n} |\hat \mu(t2^n)| \right) \qquad (n \in \bN).
\]
As $\alp < \bet_1$, we find that if $n \in G_N$, then
\begin{align}\label{A_delta bound}
\mu(A_n^y(\del_n)) \ll \del_n.
\end{align}

To pass between the two scales $\delta_n$ and $\sigma_n$, we require an inhomogeneous analogue of \cite[Lemma 2.2]{ACY}. Its statement and proof are based upon the iterative construction of $K$, which we now briefly recall, see \cite[\S 2]{ACY} for further details. For $N \in \bN$, the $N^{\mathrm{th}}$ level in the construction of the Cantor set, which we denote by $K_N$, comprises $2^N$ intervals of length~$3^{-N}$. The left endpoints of these intervals form the set $L_N$ of rationals $a/3^N$ such that $a \in [0,3^N]$ is an integer whose ternary expansion contains only the digits $0$ and $2$, and the right endpoints form the set $R_N = \{ 1 - x: x \in L_N \}$. We write $C_N = L_N \cup R_N$. The following is an inhomogeneous analogue of \cite[Lemma 2.2]{ACY}.

\begin{lemma} \label{ShiftDrop}
Let $n,M,N \in \bN$ and $\sig, \del \in \bR$ satisfy
\[
0 < \sig < \del \le 1,
\qquad
\frac{\sig}{2^{n+5}} \le 3^{-N} \le \frac{\sig}{2^n} \le 3^{-M} \le \frac{\del}{2^n},
\]
and let $y \in \bR$. Then
\[
|C_N \cap A_n^y(\sig)| \ll |C_M \cap A_n^y(2\del)|.
\]
\end{lemma}

\begin{proof} We imitate the proof of \cite[Lemma 2.2]{ACY}. By symmetry, it suffices to prove that
\begin{equation} \label{left}
|L_N \cap A_n^y(\sig)| \ll |L_M \cap A_n^y(2\del)|.
\end{equation}
Suppose $x \in L_N \cap A_n^y(\sig)$. Then $x = a/3^N$ for some integer $a \in [0,3^N)$ whose ternary expansion contains only the digits $0$ and $2$. Further, there exists an integer $b \in [0,2^n]$ such that $|x-(b+y)/2^n| < \sig/2^n$. Therefore $|L_n \cap A_n^y(\sig)|$ is bounded above by the number of integer solutions $(a,b)$ to the inequality
\[
\left| \frac a{3^N} - \frac{b+y}{2^n} \right| < \frac{\sig}{2^n}
\]
such that $a \in [0,3^N)$, $b \in [0,2^n]$, and each ternary digit of $a$ is $0$ or $2$.

We decompose $a = a_1 a_2$, where $a_1$ represents the right $M$ ternary digits of $a$ and $a_2$ represents the remaining $N-M$ digits. This reveals that $|L_n \cap A_n^y(\sig)|$ is bounded above by the number of integer solutions $(a_1,a_2,b)$ to
\begin{equation} \label{new}
\left| \frac{3^{N-M}a_1 + a_2}{3^N} - \frac {b+y} {2^n} \right| < \frac{\sig}{2^n}
\end{equation}
such that 
\[
0 \le a_1 < 3^M,
\qquad
0 \le a_2 < 3^{N-M},
\qquad
0 \le b \le 2^n,
\]
and the ternary digits of $a_1,a_2$ are all $0$ or $2$. As
\begin{equation} \label{triangle}
\left| \frac{a_1}{3^M}  - \frac{b+y}{2^n} \right| \le \left|
\frac{a_1}{3^M} + \frac{a_2}{3^N} - \frac{b+y}{2^n} \right| + \frac{a_2}{3^N} < \frac{\sig}{2^n} + \frac1{3^M} \le \frac2{3^M},
\end{equation}
we must have $a_1/3^M \in A_n^y(2\del)$ for any such solution.

Given $a_1$, the inequality \eqref{triangle} forces $b/2^n$ to lie in the interval of length $4/3^M$, and so there are at most $O(1)$ possibilities for $b$. Next, suppose we are given $a_1$ and $b$. Then, by \eqref{new}, the integer $a_2$ is forced to lie in the interval of length $3^N \sig 2^{1-n}$ centred at $3^N ((b+y)2^{-n} - 3^{N-M}a_1)$, and so there are at most $O(1)$ solutions $a_2$ to \eqref{new}. Finally, since $a_1/3^M \in L_M \cap A_n^y(2 \del)$, we conclude that there are $O(|L_M \cap A_n^y(2 \del)|)$ solutions in total. This confirms \eqref{left} and completes the proof of the lemma.
\end{proof}

Now \cite[Lemma 2.1]{ACY} and Lemma \ref{ShiftDrop}, together with \eqref{A_delta bound}, yield
\[
\mu(A_n^y(\sig_n)) \ll \frac{(\sig_n/2^n)^\gam}{(\del_n/2^n)^\gam} \mu(A_n^y(\del_n)) \ll \del_n^{1-\gam} \sig_n^\gam \qquad (n \in G_N).
\]
Hence 
\begin{align*}
\sum_{n=N}^{2N} \mu(A_n^y(\sig_n))
&\ll \sum_{n=N}^{2N} \del_n^{1-\gam} \sig_n^\gam + \sum_{n \in B_N} \sig_n^\gam
\\ &\ll \sum_{n=N}^{2N} \frac1{n^{\tau \gam + \alp(1-\gam)}} + N^{\bet_2 + \alp - \tau \gam}.
\end{align*}
In view of \eqref{constraint}, and noting that we can write
\[\sum_{n=1}^{\infty}{\mu(A_n^y(\sigma_n))} \le \sum_{k=0}^{\infty}{\sum_{n=2^k}^{2^{k+1}}{\mu(A_n^y(\sigma_n))}},\]
we finally have
\eqref{BC}, which completes the proof of Theorem \ref{MainThm}.

\providecommand{\bysame}{\leavevmode\hbox to3em{\hrulefill}\thinspace}

\end{document}